\documentclass[11pt,leqno,amscd,amssymb,verbatim, url]{amsart}
\usepackage{amsfonts,amssymb}
\usepackage{amsmath,amscd}
\usepackage{graphicx}
\usepackage{caption}
\usepackage{subcaption}
\newtheorem{thm}{Theorem}[section]
\usepackage{amsfonts,latexsym}
\usepackage{enumerate}

\newtheorem{lem}[thm]{Lemma}
\newtheorem{cor}[thm]{Corollary}
\newtheorem{conj}[thm]{Conjecture}
\newtheorem{prop}[thm]{Proposition}

\theoremstyle{definition}

\numberwithin{equation}{thm}
\usepackage{wrapfig}

\newcommand{\N}{{\mathbb N}}

\newcommand{\bC}{\mathbb{C}}

\newcommand{\fS}  {\mathfrak{S}}

\newcommand{\blist}{\begin{list}{\rom{(\roman{enumi})}}{\setlength
{\leftmarg in}{0em} \setlength{\itemindent}{7ex}
\setlength{\labelsep}{2ex}\setlength{\listparindent}{\parindent}
\usecounter{enumi}}}
\newcommand{\elist}{\end{list}}

\newtheorem{ex}[thm]{Example}

\newcommand{\inv}  {\text{inv}}

\newcommand{\des}  {\text{des}}
\newcommand{\stat}   {\text{st}}

\title{Permutation Statistics and Multiple Pattern Avoidance}
\author{Wuttisak Trongsiriwat}

\begin{document}

\begin{abstract}  
For a set of permutation patterns $\Pi$, let $F^\stat_n(\Pi,q)$ be the st-polynomial of permutations avoiding all patterns in $\Pi$. Suppose $312\in\Pi$. For a class of permutation statistics which includes inversion and descent statistics, we give a formula that expresses $F^\text{st}_n(\Pi;q)$ in terms of these st-polynomials where we take some subblocks of the patterns in $\Pi$. Using this formula, we can construct many examples of nontrivial st-Wilf equivalences. In particular, this disproves a conjecture by Dokos, Dwyer, Johnson, Sagan, and Selsor that all $\text{inv}$-Wilf equivalences are trivial.
\end{abstract}

\maketitle

{\large\section{Introduction}}

Let $\fS_n$ be the set of permutations of $[n]:=\{1,2,...,n\}$ and let $\fS=\bigcup_{n\geq 0}\fS_n$, where $\fS_0$ contains only one element $\epsilon$ - the empty permutation. For permutation $\pi,\sigma\in \fS$ we say that the permutation $\sigma$ \emph{contains} $\pi$ if there is a subsequence of $\sigma$ having the same relative order as $\pi$. 
In particular, every permutation contains $\epsilon$, and every permutation except $\epsilon$ contains $1\in\fS_1$.
For consistency, we will use the letter $\sigma$ as a permutation and $\pi$ as a pattern. We say that $\sigma$ {\it avoids} $\pi$ (or $\sigma$ is $\pi$-{\it avoiding}) if $\sigma$ does not contain $\pi$. 
For example, the permutation 46127538 contains 3142 since it contains the subsequence while the permutation 46123578 avoids 3142.
We denote by $\fS_n(\pi)$, where $\pi\in\fS$, the set of permutations $\sigma\in\fS_n$ avoiding $\pi$. More generally we denote by $\fS_n(\Pi)$, where $\Pi\subseteq \fS$, the set of permutations avoiding each pattern $\pi\in\Pi$ simultaneously, i.e. $\fS_n(\Pi) = \bigcap_{\pi\in\Pi} \fS_n(\pi)$. Two sets of patterns $\Pi$ and $\Pi'$ are called {\it Wilf equivalent}, written $\Pi\equiv\Pi'$, if $|\fS_n(\Pi)|=|\fS_n(\Pi')|$ for all integers $n\geq 0$.

Now we define the $q$-analogue of pattern avoidance using permutation statistics. A \emph{permutation statistic} (or sometimes just \emph{statistic}) is a function $\stat:\fS\rightarrow\N$, where $\N$ is the set of nonnegative integers. Given a permutation statistic $\stat$, we define the \emph{st-polynomial} of $\Pi$-avoiding permutations to be
$$ F^{\stat}_n(\Pi)=F^{\stat}_n(\Pi;q) := \sum_{\sigma \in \fS_n(\Pi)} q^{\stat(\sigma)}. $$
We may drop the $q$ if it is clear from the context. The set of patterns $\Pi$ and $\Pi'$ are said to be st-\emph{Wilf equivalent} if $F^{\stat}_n(\Pi;q)= F^{\stat}_n(\Pi';q)$ for all $n\geq 0$.

The study of $q$-analogue of pattern avoidance using permutation statistics and the st-WIlf equivalences began 2002, as initiated by Robertson, Saracino, and Zeilberger \cite{RSZ}, with the emphasis on the number of fixed points. Elizalde subsequently refined results of Robertson et al. by considering the excedance statistic \cite{Eliz} and later extended the study to cases of multiple patterns \cite{Eliz2}. A bijective proof was later given by Elizalde and Pak \cite{ElPa}. Dokos et al. \cite{DDJSS} studied pattern avoidance on the the inversion and major statistics, as remarked by Savage and Sagan in their study of Mahonian pairs \cite{SaSa}.

In this paper, we study multiple pattern avoidance on a class of permutation statistics which includes the inversion and descent statistics. The {\it inversion number} of $\sigma\in\fS_n$ is
$$\inv(\sigma) = \#\{(i,j)\in [n]^2: i<j \text{ and } \sigma(i) > \sigma(j) \}. $$ 
The {\it descent number} of $\sigma\in\fS_n$ is 
$$ \des(\sigma) = \#\{i\in [n-1]: \sigma(i)>\sigma(i+1) \}. $$
For example $\inv(3142) = \#\{(1,2),(1,3),(3,4)\} = 3$ and $\des(3142) = \#\{1,3\} = 2$.  

In [1], Dokos et al. conjectured that there are only essentially trivial inv-Wilf equivalences, obtained by rotations and reflections of permutation matrices. Let us describe these more precisely. The notations used below are mostly taken from \cite{DDJSS}.

For $\sigma\in\fS_n$, we represent it geometrically using the squares $(1,\sigma(1)),(2,\sigma(2)),...,(n,\sigma(n))$ of the $n$-by-$n$ grid, which is coordinated according to the $xy$-plane. This will be referred as the {\it permutation matrix} of $\sigma$. The diagram to the left in the Figure 1 is the permutation matrix of 46127538. In the diagram to the right, the red squares correspond to its subsequence 4173, which is an occurrence of the pattern 3142.
\begin{figure}
\begin{center}
\includegraphics[scale=0.5]{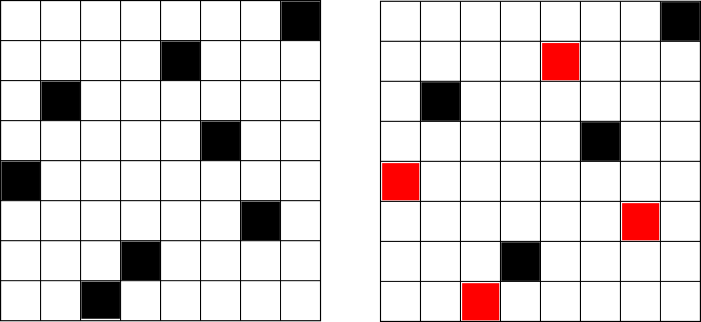}
\caption{The permutation matrix of 46127538 (left) with an occurrence of  3142 colored (right)}
\end{center}
\end{figure}

By representing each $\sigma\in\fS$ using a permutation matrix, we have an action of the dihedral group of square $D_4$ on $\fS$ by the corresponding action on the permutation matrices. We denote the elements of $D_4$ by 
$$D_4=\{R_0,R_{90},R_{180},R_{270}, r_{-1},r_{0},r_{1},r_{\infty}\},$$ 
where $R_\theta$ is the counter-clockwise rotation  by $\theta$ degrees and $r_m$ is the reflection in a line of slope $m$. We will sometimes write $\Pi^t$ for $r_{-1}(\Pi)$.  Note that $R_0,R_{180},r_{-1},$ and $r_1$ preserve the inversion statistic while the others reverse it, i.e.
$$ \inv(f(\sigma)) = \begin{cases} 
\inv(\sigma) &\text{if } f\in\{R_0,R_{180},r_{-1},r_1\},\\ 
{n\choose 2} - \inv(\sigma) &\text{if } f\in\{R_{90},R_{270},r_0,r_{\infty}\}.\end{cases} $$
It follows that $\Pi$ and $f(\Pi)$ are $\inv$-Wilf equivalent for all $\Pi\in\fS$ and $f\in\{R_0,R_{180},r_{-1},r_1\}$. We call these equivalences trivial.
With these notations, the conjecture by Dokos et al. can be stated as the following.\\

\begin{conj}(\cite{DDJSS}, conj. 2.4)
 $\Pi$ and $\Pi'$ are $\inv$-Wilf equivalent iff $\Pi=f(\Pi')$ for some $f\in\{R_0,R_{180},r_{-1},r_1\}$. 
\end{conj}

Given permutations $\pi=a_1 a_2...a_k\in\fS_k$ and $\sigma_1,...,\sigma_k\in\fS$, the {\it inflation} $\pi[\sigma_1,...,\sigma_k]$ of $\pi$ by the $\sigma_i$ is the permutation whose permutation matrix is obtained by putting the permutation matrices of $\sigma_i$ in the relative order of $\pi$; for instance, 213[123,1,21]=234165 as illustrated in Figure 2.

\begin{figure}
\begin{center}
\includegraphics[scale=0.5]{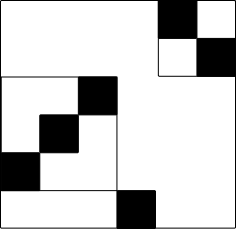}
\caption{The permutation 213[123,1,21]}
\end{center}
\end{figure}

For convenience, we write 
$$\pi_* := 21[\pi,1].$$ 
In other words, $\pi_*$ is the permutation whose permutation matrix is obtained by adding a box to the lower right corner of the permutation matrix of $\pi$.

The next proposition is one of the main results of this paper, which disproves the conjecture above. This is a special case of the corollary of the theorem \ref{main1} in the next section.

\begin{prop}
Let $\pi_1,...,\pi_r,\pi'_1,...,\pi'_r$ be permutations such that $\{312,\pi_i\} \stackrel{\inv}{\equiv} \{312, \pi'_i\}$ for all $i$. Set $\pi=\iota_r[\pi_{1*},...,\pi_{r*}]$ and $\pi'=\iota_r[\pi'_{1*},...,\pi'_{r*}]$. Then $\{312,\pi\}$ and $\{312,\pi'\}$ are also $\inv$-Wilf equivalent, i.e. $F^\inv_n(312,\pi)=F^\inv_n(312,\pi')$ for all $n$.
\end{prop}

In particular, if we set each $\pi'_i$ to be either $\pi_i$ or $\pi_i^t$, then the conditions $\{312,\pi_i\} \stackrel{\inv}{\equiv} \{312, \pi'_i\}$ are satisfied. By this construction $\Pi'$ is generally not of the form $f(\Pi)$ for any $f\in\{R_0,R_{180},r_{-1},r_1\}$. For example, the pair $\Pi=\{312, 32415\}$ and $\Pi'=\{312, 24315\}$ is a smallest example of a nontrivial $\inv$-Wilf equivalent constructed this way.


\vspace{3mm}
{\large\section{Avoiding two patterns}}

In this section, we give a recursive formula for the polynomial $F^\stat_n(\Pi)$ when $\Pi$ consists of 312 and another permutation $\pi$. Then we will present its corollary, which gives a construction of nontrivial st-Wilf equivalences. The idea in the proof of the main theorem is similar to those in \cite{MaVa}.

Suppose $\sigma\in\fS_{n+1}(312)$ with $\sigma(k+1)=1$. Then, for every pair of indexes $(i,j)$ with $i<k+1<j$, we must have $\sigma(i)<\sigma(j)$; otherwise $\sigma(i) \sigma(k+1) \sigma(j)$ forms a pattern 312 in $\sigma$. So $\sigma$ can be written as $\sigma=213[\sigma_1,1,\sigma_2]$ with $\sigma_1\in\fS_k$ and $\sigma_2\in\fS_{n-k}$. In the rest of the paper, we will always consider $\sigma$ in its inflation form.

For the rest of the paper, we assume that the permutation statistic $\stat:\fS_n\rightarrow\N$ satisfies 
\[
\stat(\sigma) = f(k,n-k)+\stat(\sigma_1)+\stat(\sigma_2) \tag{\textdagger}
\]
for some function $f:\N^2\rightarrow\N$ that is independent of the statistic st.
Some examples of such statistics are the inversion number, the descent number, and the number of occurrences of the consecutive pattern 213:
$$
\underline{213}(\sigma) = \#\{i\in [n-2]: \sigma(i+1)<\sigma(i)<\sigma(i+2)\}.
$$
For these mentioned statistics, we have
\begin{align*}
\inv(\sigma) &= k+\inv(\sigma_1)+\inv(\sigma_2),\\
\des(\sigma) &= 1-\delta_{0,k}+\des(\sigma_1)+\des(\sigma_2),\\
\underline{213}(\sigma)  &= (1-\delta_{0,k})(1-\delta_{k,n})+\underline{213}(\sigma_1)+\underline{213}(\sigma_2).
\end{align*}

For a pattern $\pi$, it will be more beneficial to consider $\pi$ in its \emph{block decomposition} as stated in the following proposition.\\

\begin{prop}
Every 312-avoiding permutation $\pi\in\fS_n(312)$ can be written uniquely as
$$ \pi=\iota_r[\pi_{1*},...,\pi_{r*}] $$
where $r\geq0$ and $\pi_i\in\fS(312)$. Here $\iota_r$ denotes the identity element $12...r$ of $\fS_r$.
\end{prop}

\begin{proof}
The uniqueness part is trivial. The proof of existence of $\pi_1,...,\pi_r$ is by induction on $n$. If $n=0$, there is nothing to proof. Suppose the result holds for $n$. Suppose that $\pi(k+1)=1$. Then $\pi=213[\pi_1,1,\pi']=12[\pi_{1*},\pi']$ where $\pi_1\in\fS_{k}(312)$ and $\pi'\in\fS_{n-k}(312)$. Applying the inductive hypothesis on $\pi'$, we are done.
\end{proof}

Suppose that $\pi\in\fS_n(312)$ has the block decomposition $\pi=\iota_r[\pi_{1*},...,\pi_{r*}]$. For $1\leq i\leq r$, we define $\underline{\pi}(i)$ and $\overline{\pi}(i)$ as
$$
\underline{\pi}(i)= 
\begin{cases}
\pi_1     &\text{ if } i=1, \\
\iota_i[\pi_{1*},...,\pi_{i*}] &\text{ otherwise,} 
\end{cases}
$$ 
and
$$
\overline{\pi}(i)= \iota_{r-i+1}[\pi_{i_*},...,\pi_{r_*}].
$$

Let $\Pi=\{312,\pi\}$. If $\pi$ contains the pattern 312, then every permutation avoiding 312 will automatically avoid $\pi$, which means $F^\inv_n(\Pi)=F^\inv_n(312)$. So for the rest of this paper we will assume that every pattern besides 312 in a set of patterns $\Pi$ avoids 312. 
We will need the following lemma which gives a recursive condition for a permutation $\sigma=213[\sigma_1,1,\sigma_2]\in\fS(312)$ to avoid $\pi$, in terms of $\sigma_1, \sigma_2,$ and the blocks $\pi_{i*}$ of $\pi$. 

\begin{lem}{\label{lem1}} Let $\sigma=213[\sigma_1,1,\sigma_2], \pi=\iota_r[\pi_{1*},...,\pi_{r*}]\in \fS(312)$. Then $\sigma$ avoids $\pi$ if and only if the condition
\begin{itemize}
\item[$(C_i):$] $\sigma_1$ avoids $\underline{\pi}(i)$ and $\sigma_2$ avoids $\overline{\pi}(i)$.
\end{itemize}
hold for some $i\in [r]$
\end{lem}

\begin{proof} 
First, suppose that $\sigma$ contains $\pi$. Let $j$ be the largest number for which $\sigma_1$ contains $\underline{\pi}(j)$. Then $\sigma_2$ must contain $\overline{\pi}(j+1)$. So $\sigma_1$ contains $\underline{\pi}(i)$ for all $i\leq j$, and $\sigma_2$ contains $\overline{\pi}(i)$ for all $i>j$. Thus none of the $C_i$ holds.

On the other hand, suppose that there is a permutation $\sigma\in\fS(312)$ that avoids $\pi$ but does not satisfy any $C_i$. This means, for every $i$, either $\sigma_1$ contains $\underline{\pi}(i)$ or $\sigma_2$ contains $\overline{\pi}(i)$.
Let $j$ be the smallest number such that $\sigma_1$ does not contain $\underline{\pi}(j)$. Note that $j$ exists and $j>1$ since $j=1$ implies $\sigma_2$ contains $\bar{\pi}(1)=\pi$, a contradiction.
Since $\sigma_1$ does not contain $\underline{\pi}(j)$, $\sigma_2$ must contain $\overline{\pi}(j)$ (by $C_j$). 
But since $\sigma_1$ contain $\underline{\pi}(j-1)$ by minimality of $j$, we have found a copy of $\pi$ in $\sigma$ with $\underline{\pi}(j-1)$ from $\sigma_1$ and $\overline{\pi}(j)$ from $\sigma_2$, a contradiction. 
(For $j=2$, the number 1 in $\sigma$ together with $\pi_1$ in $\sigma_1$ give $\pi_{1*}$.)
\end{proof} 

Before presenting the main result, we state a technical lemma regarding the M\"obius function of a certain poset. See, for example, the chapter 3 of \cite{Stan} for definitions and terminologies about posets and the general treatment of the subject.

Let $\bf{r}$ be the chain of $r$ elements $0<1<...<r-1$. Let $L_r$ be the poset obtained by taking the elements of $\bf{r} \times \bf{r}$ of rank $0$ to $r-1$, i.e. the elements of $L_r$ are the lattice points $(a,b)$ where $a,b\geq 0$ and $a+b<r$. For instance, $L_5$ is the poset shown in Figure 3. We denote the minimal element $(0,0)$ in $L_r$ by $\hat{0}$

\begin{figure}
\begin{center}
\includegraphics[scale=0.5]{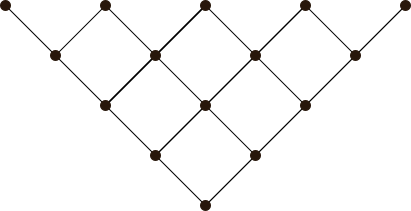}
\caption{The poset $L_5$}
\label{fig:L5}
\end{center}
\end{figure}

Let $\hat{L}_r$ be the poset $L_r$ with the unique maximum element $\hat{1}$ adjoined. Note that for every element $a\in \hat{L}_r$ the up-set $U(a) := \{ x\in \hat{L}_r : x\geq a \}$ of $a$ is isomorphic to $\hat{L}_{r-l(a)}$ where $l(a)$ is the rank of $a$ in $L_r$. So to understand the M\"obius function $\mu_{\hat{L}}$ on these $\hat{L}_r$, it suffices to know the value of $\mu_{\hat{L}_r}(\hat{0},\hat{1})$ for every $r$, which is given by the following lemma. The proof is omitted since it is by a straightforward calculation.

\begin{lem}
We have
$$ \mu_{\hat{L}_r}(\hat{0},\hat{1}) =
\begin{cases}
(-1)^r, &\text{ if } r=1,2,\\
0, &\text{ otherwise.}
\end{cases} $$ 
\end{lem}

We now state the main theorem of this section.

\begin{thm}{\label{main1}} 
Let $\Pi=\{312,\pi\}$. Suppose that the statistic $\stat:\fS\rightarrow\N$ satisfies the condition (\textdagger). Then $F^\stat_n(\Pi;q)$ satisfies
\begin{align*}\label{eq:maineq}
 F^\stat_{n+1}(\Pi;q) = \sum_{k=0}^n q^{f(k,n-k)}  \Bigg[ \sum_{i=1}^r & F^\stat_k(312,\underline{\pi}(i))\cdot F^\stat_{n-k}(312, \overline{\pi}(i)) \\
 &-\sum_{i=1}^{r-1}  F^\stat_k(312,\underline{\pi}(i))\cdot F^\stat_{n-k}(312,\overline{\pi}(i+1)) \Bigg], \tag{*}
\end{align*}
for all $n\geq 0$, where $F^\stat_0(\Pi;q) = 0$ if $\pi=\epsilon$, and $1$ otherwise.
\end{thm}

\begin{proof} For $ k\in \{0,1,...,n\}$, we write $\fS^k_{n+1}(\Sigma)$ where $\Sigma\subset\fS$ to denote the set of permutations $\sigma\in\fS_{n+1}(\Sigma)$ such that $\sigma(k+1)=1$. In particular, 
$$\fS^k_{n+1}(312) = \{\sigma=213[\sigma_1,1,\sigma_2]: \sigma_1\in \fS_k(312) \text{ and } \sigma_2\in\fS_{n-k}(312)\}.$$

\noindent Fix $k$, and let $A_i (i\in[r])$ be the set of permutations in $\fS^k_{n+1}(312)$ satisfying the condition $C_i$. So $\fS^k_{n+1}(\Pi)=A_1\cup A_2 \cup \cdots \cup A_n =: A.$ Observe that if $i_1<...<i_k$ then
$$ A_{i_1}\cap A_{i_2}\cap\cdots\cap A_{i_k} = A_{i_1}\cap A_{i_k} =: A_{i_1,i_k}. $$
This is since satisfying the conditions $C_{i_1},...,C_{i_k}$ is equivalent to satisfying the conditions $C_{i_1}$ and $C_{i_k}$.

Let $P$ be the intersection poset of $A_1,...,A_n$, which consists of the unique maximal element $A$, the $A_i$,  and $A_{i,j}$ for $1\leq i < j\leq r$. So $P$ is isomorphic to the set $\hat{L}_r$. 
Thus the M\"obius function $\mu_P(T,A)$ for $T\in P$ is given by
$$ \mu_P(T,A) = 
\begin{cases}
1 &\text{ if } T=A \text{ or } A_{i,i+1} \text{ for some } i,\\
-1 &\text{ if } T=A_i \text{ for some } i,\\
0 &\text{ otherwise}. 
\end{cases} $$
For $T\in P$, we define $g:P\rightarrow \bC(x: x\in A)$ by
$$ g(T)=\sum_{x\in T} x.$$
So the M\"obius inversion formula (\cite{Stan}, section 3.7) implies that
\begin{align*}
g(A) &= -\sum_{T < A} \mu_P(T,A) g(T)\\
&=\sum_{i=1}^r g(A_i) - \sum_{i=1}^{r-1} g(A_i \cap A_{i+1}).
\end{align*}

\noindent By mapping $\sigma\mapsto q^{\stat(\sigma)}$ for all $\sigma\in A$, $g(A)$ is sent to $F^\stat_{n+1,k}(\Pi;q) := \sum_{\sigma\in\fS^k_{n+1}(\Pi)} q^{\stat(\sigma)}$. Hence,

\begin{align*}
F^\stat_{n+1,k}(\Pi;q) 
 &= \sum_{i=1}^r \sum_{\sigma\in A_i} q^{\stat(\sigma)} - \sum_{i=1}^{r-1} \sum_{\sigma\in A_i\cap A_{i+1}} q^{\stat(\sigma)} \\
 &= q^{f(k,n-k)} \left[  \sum_{i=1}^r  \sum_{\sigma\in A_i} q^{\stat( \sigma_1) + \stat(\sigma_2)} - \sum_{i=1}^{r-1} \sum_{\sigma\in A_i\cap A_{i+1}} q^{\stat(\sigma_1) + \stat(\sigma_2)} \right],
\end{align*}
where the second equality is obtained from the condition (\textdagger).\\
Note that $\sigma\in A_i$ iff $\sigma_1$ avoids $\underline{\pi}(i)$ and $\sigma_2$ avoids $\overline{\pi}(i)$, and $\sigma\in A_i\cap A_{i+1}$ iff $\sigma_1$ avoids $\underline{\pi}(i)$ and $\sigma_2$ avoids $\overline{\pi}(i+1)$. Thus

$$\sum_{\sigma\in A_i} q^{\stat(\sigma_1) + \stat(\sigma_2)} = F^\stat_{k}(312,\underline{\pi}(i))\cdot F^\stat_{n-k}(312,\overline{\pi}(i)) $$
and
$$\sum_{\sigma\in A_i\cap A_{i+1}} q^{\stat(\sigma_1) + \stat(\sigma_2)} = F^\stat_{k}(312,\underline{\pi}(i))\cdot F^\stat_{n-k}(312,\overline{\pi}(i+1)).$$

\noindent Therefore
\begin{align*}
F^\stat_{n+1,k}(\Pi;q) = q^{f(k,n-k)} \Bigg[  \sum_{i=1}^r  & F^\stat_{k}(312,\underline{\pi}(i))\cdot F^\stat_{n-k}(312,\overline{\pi}(i))\\
 &- \sum_{i=1}^{r-1} F^\stat_{k}(312,\underline{\pi}(i))\cdot F^\stat_{n-k}(312,\overline{\pi}(i+1)) \Bigg].
\end{align*}
Summing the equation above from $k=0$ to $n$, we get the stated result. 
\end{proof}

\begin{ex}(A $q$-analoque of odd Fibonacci numbers)
It is well-known that the permutations avoiding 312 and 1432 are counted by the Fibonacci numbers $F_{2n+1}$ assuming $F_1=F_2=1$ (see \cite{West} for example). Let $A_n=F_{2n-1}$. It can be shown that the $A_n$ satisfy
$$ A_{n+1} = A_n + \sum_{k=0}^{n-1} 2^{n-k-1} A_k.$$
Theorem \ref{main1} gives $q$-analogues of this relation. Here, we will consider the inversion statistic inv.

Let $\pi=1432=12[\epsilon_*,21_*]$ and $\Pi=\{312,\pi\}$. Since $\underline{\pi}(1) = \epsilon$ and $F^{\inv}_{n}(312,\epsilon)=0$ for all $n$, theorem \ref{main1} implies
\begin{align*}
F^{\inv}_{n+1}(\Pi) 
&= \sum_{k=0}^{n} F^{\inv}_{k}(\Pi) F^{\inv}_{n-k}(312,321)\\
&= q^n F^\inv_n(\Pi) + \sum_{k=0}^{n-1} q^k (1+q)^{n-k-1} F^\inv_{k}(\Pi),
\end{align*}
where the last equality is by \cite{DDJSS}, Proposition 4.2.
\end{ex}

\begin{cor}{\label{cor1}}
Let $\pi_1,...,\pi_r,\pi'_1,...,\pi'_r$ be permutations such that $\{312,\pi_i\} \stackrel{\stat}{\equiv} \{312, \pi'_i\}$ for all $i$. Set $\pi=\iota_r[\pi_{1*},...,\pi_{r*}]$ and $\pi'=\iota_r[\pi'_{1*},...,\pi'_{r*}]$. Then $\{312,\pi\}$ and $\{312,\pi'\}$ are also $\stat$-Wilf equivalent, i.e. $F^\stat_n(312,\pi)=F^\stat_n(312,\pi')$ for all $n$.
\end{cor}
\begin{proof}
The proof is by induction over $n$. 
If $n=0$, then the statement trivially holds.
Now suppose the statement holds up to $n$. Then for $0\leq k\leq n$ and $1\leq i\leq r$, we have $F^\stat_k(312,\underline{\pi}(i)) = F^\stat_k(312,\underline{\pi'}(i))$ and $F^\stat_{n-k}(312,\overline{\pi}(i))=F^\stat_{n-k}(312,\overline{\pi'}(i)).$ Hence $F^\stat_{n+1}(312,\pi)=F^\stat_{n+1}(312,\pi')$ by comparing the terms on the right hand side of (*).
\end{proof}

As mentioned at the end of section one, for the inversion statistic we can choose take each $\pi'_i$ to be either $\pi_i$ or $\pi_i^t$. Indeed, this construction works for every statistic $\stat$ satisfying (\textdagger) and that $\stat(\sigma)=\stat(\sigma^t)$ for all $\sigma\in\fS(312)$.
Besides the inversion statistic, the descent statistic also possesses this property.
To justify this fact, we write $\sigma=213[\sigma_1,1,\sigma_2]\in\fS(312)$ where $\sigma_1,\sigma_2\in\fS$.
Observe that $\sigma^t=132[\sigma_2^t,\sigma_1^t,1]$ and 
$$\des(\sigma^t) = \des(\sigma_2^t) + \des(\sigma_1^t) + (1-\delta_{0,k}) $$
where $k=|\sigma_1^t|=|\sigma_1|$. The proof then proceeds by induction on $n=|\sigma|$.
Note that, however, it is not true in general that the matrix transposition preserves the descent number.


{\large\section{Generalization}}

In this section, we will generalize the theorem \ref{main1} to the case when $\Pi$ consists of 312 and other patterns. Then we give a generalized version of the corollary \ref{cor1}.

\begin{lem}
Let $L$ be the poset $L_{r_1}\times\cdots\times L_{r_m}$ and $\hat{L}$ the poset $L$ adjoined by the maximal element $\hat{1}$. Let $\mu=\mu_{\hat{L}}$ be the M\"obius function on $\underline{L}$. Then
$\mu(\hat{0},\hat{1}) = 0$ unless each $r_i\in\{1,2\}$, in which case 
$\mu(\hat{0},\hat{1}) = (-1)^{|S|+1}$, where $|S|=\{i:r_i=2\}$.
\end{lem}

\begin{proof}
Let $a=(a_1,...,a_m)\in L$. Then $\mu(\hat{0},a)=\prod_{i=1}^m \mu_i(\hat{0},a_i)$, where $\mu_i$ is the M\"obius function on $L_{r_i}$. So
$$ \mu(\hat{0},\hat{1})=-\sum_{a\in L} \mu(\hat{0}, a) = -\prod_{i=1}^m \left( \sum_{a_i\in L_{r_i}} \mu_{i}(\hat{0},a_i) \right).$$
Note that for $r\geq 3$ $\mu_{L_r}(\hat{0},a)=0$ unless $a\in\{(0,0),(1,0),(0,1),(1,1)\},$ in which cases $\mu_{L_r}(\hat{0},a)$ is 1,-1,-1,1, respectively. So $\sum_{a\in L_r} \mu_{L_r}(\hat{0},a_i) = 0$ unless $r=1,2$.
In the case of $r=1,2$, it can easily be checked that $\sum_{a\in L_r} \mu_{L_r}(\hat{0},a_i) = 1$ if $r=1$ and -1 if $r=2$. So if $r_i\geq 3$ for some $i$, then $\mu(\hat{0},\hat{1})=0$. If each $r_i\in\{1,2\}$, then each index $i$ for which $r_i=2$ contributes a -1 to the product on the right hand side of the previous equation. Thus $\mu(\hat{0},\hat{1}) = (-1)^{|S|+1}$.
\end{proof}

The following theorem is a generalization of (\ref{main1}). For convenience, we introduce the following notations.
Let $\Pi=\{312,\pi^{(1)},...,\pi^{(m)}\}$ where $\pi^{(j)}=\iota_{r_j}[(\pi^{(j)}_1)_*,...,(\pi^{(j)}_{r_j})_*]$.
For $I=(i_1,...,i_m)$, we define
$$\underline{\Pi}_I = \{312,\underline{\pi^{(1)}}(i_1),...,\underline{\pi^{(m)}}(i_m)\}$$
and
$$\overline{\Pi}_I = \{312,\overline{\pi^{(1)}}(i_1),...,\overline{\pi^{(m)}}(i_m)\}. $$

\begin{thm}{\label{main2}} 
Let $\Pi=\{312,\pi^{(1)},...,\pi^{(m)}\}$ where $\pi^{(i)}=\iota_{r_i}[(\pi^{(i)}_1)_*,...,(\pi^{(i)}_{r_i})_*]$. Then $F^{\stat}_0(\Pi)= 0$ if some $\pi_i=\epsilon$ and $1$ otherwise, and for $n\geq 1$
$$
F^{\stat}_{n+1}(\Pi;q) 
 = \sum_{k=0}^n q^{f(k,n-k)} \Bigg[ \sum_{S\subseteq [m]} (-1)^{|S|}  \sum_{\substack{I=(i_1,...,i_m): \\1\leq i_j\leq r_j-\delta_j}} F^{\stat}_k(\underline{\Pi}_I)\cdot F^\stat_{n-k}(\overline{\Pi}_{I+\delta}) \Bigg],
$$
Here $\delta_j = 1$ if $j\in S$ and $0$ if $j\notin S$.
\end{thm}

\begin{proof}
Recall that by (\ref{lem1}) $\sigma=213[\sigma_1,1,\sigma_2]\in\fS(312)$ avoids $\pi^{(j)}$ iff $\sigma$ satisfies the condition

$(C^j_i)$: $\sigma_1$ avoids $\underline{\pi^{(j)}}(i)$ and $\sigma_2$ avoids $\overline{\pi^{(j)}}(i)$\\
for some $i\in [r_j]$.
So $\sigma\in\fS(312)$ belongs to $\fS(\Pi)$ if for all $j$, there is an $i\in [r_j]$ for which $\sigma$ satisfies $(C^j_i)$.
Fix $k$ and let $\fS_{n+1}^k(312)$ be as in the proof of (\ref{main1}). Let $A_i^j$ be the set of permutations in $\fS_{n+1}^k(312)$ avoiding $\pi^{(j)}$ and satisfying the condition $(C^j_i)$. For $I=(i_1,...,i_m)\in [r_1]\times [r_2] \times \cdots \times [r_m]$, we define
$$ A_I = A_{i_1,i_2,...,i_m} := A^1_{i_1} \cap A^2_{i_2} \cap A^m_{i_m}. $$ So $\fS^k_{n+1}(\Pi)$ is the union
$$\fS^k_{n+1}(\Pi) = \bigcup_{i_1,...,i_m} A_{i_1,i_2,...,i_m}, $$
where the union is taken over all $m$-tuples $I=(i_1,...,i_m)$ in $[r_1]\times [r_2] \times \cdots \times [r_m]$.
Let $\hat{P}_j$ be the intersection poset of the $A_1^j,...,A^j_{r_j}$, and let $P_j$ be the poset $\hat{P}_j\setminus\{\hat{1}\}$, where $\hat{1}=A_1^j\cup\cdots\cup A^j_{r_j}$ is the unique maximum element of $\hat{P}_j$.
Recall that $P_j$ is isomorphic to $L_{r_j}$.
Let $P$ be the intersection poset of the $A_I$. 
The elements of $P$ are the unique maximal element $A=\fS^k_{n+1}(\Pi)$ and 
$$ T=T^1\cap T^2 \cap\cdots\cap T^m,$$
where each $T^j$ is an element of $P_j$.  Thus $P$ is isomorphic to $L_{r_1}\times\cdots\times L_{r_m}$ with the unique maximum element $\hat{1}$ adjoined. 
For $S\subseteq [n]$, we say that $T\in P$ has type $S$ if
$T^j=A^j_i$ for some $i$ when $j\notin S$ and $T^j=A^j_i\cap A^j_{i+1}$ for some $i$ when $j\in S$. 
Using the previous lemma, we see that the M\"obius function on $P$ for $T=T^1\cap T^2 \cap \cdots\cap T^m \neq \hat{1} $ is 
$$ \mu(T,A) = 
\begin{cases}
 (-1)^{|S|+1}, &\text{ if } T \text{ has type } S,\\
 0, & \text{ otherwise}.
\end{cases}
$$

\noindent For $T\in P$, we define $g:P\rightarrow \bC(x: x\in A)$ by
$ g(T)=\sum_{x\in T} x,$ so that
\begin{align*}
g(A) &= -\sum_{T < A} \mu(T,A) g(T)\\
&= \sum_{S\subseteq [n]} (-1)^{|S|} \sum_{T \text{ has type } S} g(T)
\end{align*}
by the M\"obius inversion formula. 
Now, by definition of type $S$, we have
$$ \sum_{T \text{ has type } S} g(T) = \sum_{\substack{i_1,...,i_m:\\ 1\leq i_j\leq r_j-\delta_j}} g\left(\bigcap_{j\notin S} A^j_{i_j}  \ \cap \bigcap_{j\in S} (A^j_{i_j}\cap A^j_{i_j+1})\right), $$
where $\delta_j=1$ if $j\in S$ and $0$ if $j\notin S$. 
Recall that $\sigma\in A^j_{i_j}$ iff $\sigma_1$ avoids $\underline{\pi}^{(j)}(i_j)$ and $\sigma_2$ avoids $\pi^{(j)}(i_j)$, and $\sigma\in A^j_{i_j}\cap A^j_{i_j+1}$ iff $\sigma_1$ avoids $\underline{\pi}^{(j)}(i_j)$ and $\sigma_2$ avoids $\pi^{(j)}(i_j+1)$. Therefore, by mapping $\sigma\mapsto q^{\stat(\sigma)}$, we have

\begin{align*} 
g\left(\bigcap_{j\notin S} A^j_i  \ \cap \bigcap_{j\in S} (A^j_i\cap A^j_{i+1})\right) \mapsto\ & q^{f(k,n-k)} F^\stat_k(312,\underline{\pi^{(1)}}(i_1),...,\underline{\pi^{(m)}}(i_m))\\  
&\cdot F^\stat_{n-k}(312,\overline{\pi^{(1)}}(i_1+\delta_1),...,\overline{\pi^{(m)}}(i_m+\delta_m)).
\end{align*} 
Thus 
\[
F^{\stat}_{n+1,k}(\Pi;q)
 = q^{f(k,n-k)} \Bigg[ \sum_{S\subseteq [m]} (-1)^{|S|}  \sum_{\substack{i_1,...,i_m:\\ 1\leq i_j\leq r_j-\delta_j}} F^\stat_k(\underline{\Pi_I})\cdot F^\stat_{n-k}(\overline{\Pi}_{I+\delta}) \Bigg].\]
The theorem then follows by summing $F^{\stat}_{n+1,k}(\Pi;q)$ over $k$ from $1$ to $n$.
\end{proof}


\begin{ex}
Let $\Pi=\{312,\pi^{(1)},\pi^{(2)}\}$ where $\pi^{(1)}=2314=12[12_*,\epsilon_*]$ and $\pi^{(2)}=2143=12[1_*,1_*]$. We want to compute $a_n= F^\inv_n(\Pi)$ by using the theorem \ref{main2}. There are four possibilities of $S\subseteq\{1,2\}$, and for each possibility the following table shows the appearing terms.

\begin{center}
\begin{tabular}{ l l l }
$S=\emptyset$: 
& $F^\inv_k(312,12,1)\cdot F^\inv_{n-k}(\Pi)$ & $=\delta_{0,k}\cdot a_{n-k}$ \\
& $F^\inv_k(312,2314,1)\cdot F^\inv_{n-k}(312,1,2143)$ & $=\delta_{0,k}\cdot\delta_{0,n-k}$ \\
& $F^\inv_k(312,12,2143)\cdot F^\inv_{n-k}(312,2314,21)$ & =1 \\
& $F^\inv_k(\Pi)\cdot F^\inv_{n-k}(312,1,21)$ & $=\delta_{0,n-k}\cdot a_k$\\
$S=\{1\}$:
& $F^\inv_k(312,12,1)\cdot F^\inv_{n-k}(312,1,2143)$ &= $\delta_{0,k}\cdot\delta_{0,n-k}$\\
& $F^\inv_k(312,12,2143)\cdot F^\inv_{n-k}(312,1,21)$ &= $\delta_{0,n-k}$\\
$S=\{2\}$:
& $F^\inv_k(312,12,1)\cdot F^\inv_{n-k}(312,2314,21)$ &= $\delta_{0,k}$\\
& $F^\inv_k(312,2314,1)\cdot F^\inv_{n-k}(312,1,21)$ &= $\delta_{0,k}\cdot\delta_{0,n-k}$\\
$S=\{1,2\}$:
& $F^\inv_k(312,21,1)\cdot F^\inv_{n-k}(312,1,21)$ &= $\delta_{0,k}\cdot\delta_{0,n-k}$
\end{tabular}
\end{center}

\noindent Here $\delta$ is the Kronecker delta function. Thus
\begin{align*}
a_{n+1} &= \sum_{q=0}^n q^k\left[ \delta_{0,k} a_{n-k} + \delta_{0,n-k}\cdot a_k + 1 - \delta_{0,k} - \delta_{0,n-k} \right]\\
&= (1+q^n)a_n+\frac{1-q^{n+1}}{1-q} - (1+q^n)\\
&= (1+q^n)a_n+q\left(\frac{1-q^{n-1}}{1-q}\right).
\end{align*}
In particular, by setting $q=1$ we get $a_{n+1}=2a_n+n-1$ with $a_0=a_1=1$. Thus 
$$|\fS_n(312,2314,2143)| = 2^n-n.$$ 
\end{ex}

The following corollary of the theorem \ref{main2} is a generalization of the corollary \ref{cor1}.
It can be proved using a similar argument to that of \ref{cor1}, so we will omit the proof.

\begin{cor}{\label{cor2}}
Let $\pi^{(j)}_i,\pi'^{(j)}_i, 1\leq j\leq m, 1\leq i\leq r_m,$ be permutations such that 
$$\{312,\pi^{(1)}_{i_1},...,\pi^{(m)}_{i_m}\} \stackrel{\stat}{\equiv} \{312,\pi'^{(1)}_{i_1},...,\pi'^{(m)}_{i_m}\}$$ 
for all $m$-tuples $I=(i_1,...,i_m)\in [r_1]\times...\times [r_m]$. Set $\pi^{(j)}=\iota_r[\pi^{(j)}_{1*},...,\pi^{(j)}_{r_j*}]$ and $\pi'^{(j)}=\iota_r[\pi'^{(j)}_{1*},...,\pi'^{(j)}_{r_j*}]$. Then $\Pi=\{312,\pi^{(1)},...,\pi^{(m)}\}$ and $\Pi'=\{312,\pi'^{(1)},...,\pi'^{(m)}\}$ are $\stat$-Wilf equivalent.
\end{cor}


\end{document}